\documentclass[12pt]{article}

\usepackage{amssymb}
\usepackage{amsmath}
\usepackage{yfonts}
\usepackage{amsthm}
\usepackage{amsfonts}
\usepackage{graphicx}	
\usepackage{float}	
\usepackage{xcolor}
\usepackage{soul,color}
\usepackage{hyperref}
\usepackage{geometry}
\usepackage{orcidlink}
\usepackage{doi}

\theoremstyle{definition}
\newtheorem{defi}{Definition}[section]

\theoremstyle{remark}

\theoremstyle{plain}

\newtheorem{lem}[defi]{Lemma}
\newtheorem{thm}[defi]{Theorem}

\newtheorem{rem}[defi]{Remark}

\title{Emergent order spectrum for transitive homeomorphisms}

\author{F. Ciavattini$^1$\orcidlink{0009-0007-8101-7929}, 
M. Farotti\footnote{Corresponding author; email to: marco.farotti@unicam.it.} $^2$\orcidlink{0009-0001-5000-2827}, 
C. Lucamarini$^3$\orcidlink{0009-0005-1344-1790}}
\date{
{\normalsize
$^1\;$\footnotesize{\textit{Doctoral School in Mathematics, University of Rome ``Tor Vergata''\\
Via della Ricerca Scientifica, 1, Roma, 00133, Italy}}\\
\vspace{3mm}
$^2\;$\footnotesize{\textit{Doctoral School in Computer Science and Mathematics, University of Camerino,\\ Via Madonna delle Carceri, 9, Camerino, 62032, Italy}}\\
\vspace{3mm}
$^3\;$\footnotesize{\textit{Mathematics Division, School of Sciences and Technology, University of Camerino,\\ Via Madonna delle Carceri, 9, Camerino, 62032, Italy}}
}}

\begin{document}

\maketitle
            
\begin{abstract}
\noindent The \emph{Emergent Order Spectrum} $\Omega(x,y)$ is a topological invariant of dynamical systems providing order-types induced by the limit order of order-compatible nested $\varepsilon_n$-chains (with $\varepsilon_n\to 0$) from $x$ to $y$. In this paper, we investigate how rich these spectra can be under natural dynamical hypotheses. 
For a transitive homeomorphism $f$ of a compact metric space $X$ without isolated points and of cardinality $\mathfrak{c}$, we show that the global spectrum $\Omega_f(X^2)$ is universal at the countable scattered level: every countable scattered order-type together with the order-type of the rationals appears in $\Omega_f(X^2)$. More precisely, there exists a comeagre subset $\mathcal{M}\subseteq X^2$ such that, for every $(x,y)\in\mathcal{M}$, the individual spectrum $\Omega_f(x,y)$ already realizes all countably infinite scattered order-types; moreover, the order-type of the rationals belongs to $\Omega_f(x,y)$ for every pair $(x,y)\in X^2$.

\vspace{3mm}
\noindent \textbf{Keywords:} Topological Dynamics; Chain Recurrence; Scattered Linear Orderings; Conjugacy Invariants.

\noindent \textbf{MSC2020:} 37B20; 37B35; 06A05.
\end{abstract}

\section{Introduction}

The structure of attractors and the decomposition of a topological dynamical system in gradient-like and chain-transitive subsets are best understood in terms of $\varepsilon$–chains (or pseudo-orbits). The \emph{chain relation} $\mathcal{C}$, indeed, is the basis for both Conley’s decomposition theory \cite{co78} and Akin’s structural approach to attractors (see, e.g., \cite[Theorem~2.68]{kurka03}).
Topological transitivity, on its part (by which here we mean the existence of a point with a dense orbit),  is a dynamical property characterizing many important systems or subsystems, and has been studied deeply in the classical continuous case (see, e.g., \cite{Cran_,Velle_,Wang,Ash__}) and also for discontinuous maps (see, e.g., \cite{Darb_,Darb__,Dev_,Cor_}).

In this work, we investigate the chain relation in the case of transitive maps.  
Specifically, we explore some structural properties of the invariant chain-related backbone of transitive homeomorphisms of compact spaces.

Throughout the paper, by \emph{compact dynamical system} (or simply dynamical system) we mean a pair $(X,f)$ where $X$ is a compact, metrizable space and $f$ is a self-map of $X$.
For $x\in X$, we denote by $\mathcal{O}^+(x)$ and $\mathcal{O}^-(x)$, respectively, the \emph{forward} and \emph{backward orbit} of $x$ under $f$, that is 
$$
\mathcal{O}^+(x)=\{f^{k}(x):k\in\mathbb{N}\}\quad,\quad \mathcal{O}^-(x)=\{f^{k}(x):k\in\mathbb{Z}\setminus \mathbb{N}_0\},
$$ 
where by $\mathbb{N}$ and $\mathbb{N}_0$ we denote, respectively, the sets of positive and non-negative integers.
The two-sided orbit of $x$, or more simply the \emph{orbit} of $x$, is the set $$
\mathcal{O}(x)=\mathcal{O}^+(x)\cup\{x\}\cup \mathcal{O}^-(x).
$$
\begin{defi}
    Let $(X, f)$ be a dynamical system with $f$ a homeomorphism. We say that $f$ is \emph{transitive} if there exists a point $x\in X$ whose orbit is dense in $X$, that is, $\overline{\mathcal{O}(x)}=X$. 
\end{defi}
\begin{rem}\label{rem_trans}
    Notice that if $(X, f)$ is a dynamical system with $f$ a transitive homeomorphism and $X$ has no isolated points, then the set 
    $$
    \mathcal{I}=\{x\in X\,|\,\overline{\mathcal{O}^-(x)}=\overline{\mathcal{O}^+(x)}=X\}
    $$
    of points whose forward and backward orbits are both dense in $X$ is comeagre in $X$ (see, for instance, \cite[pag.~70]{oxtoby}). 
\end{rem}

Let us recall the notions of $\varepsilon$-chain and chain relation.

\begin{defi}\label{chain}
    Let $(X, f)$ be a dynamical system. Given two points $x, y \in X$ and $\varepsilon > 0$, an \emph{$\varepsilon$-chain} from $x$ to $y$ is an indexed, finite sequence of points of $X$, that is a map
    $$C:\{0,1,\dots,m\}\longrightarrow X,$$
    with $m\in\mathbb{N}$, such that, setting $x_i=C(i)$, we have:
    \begin{itemize}
        \item[i)] $x_0 = x$ and $x_m = y$,
        \item[ii)] $d(f(x_i), x_{i+1})<\varepsilon$, for every $i=0,1,\dots,m-1$.
    \end{itemize}
    By $\widehat{C}$, we denote the support of the chain $C$, that is, $$\widehat{C}=C(\{0,1,\dots,m\})=\{x_0,\dots,x_m\}.$$ 
    With an abuse of notation, we may indicate a chain $C$ by writing simply $$C:x_0,\,x_1,\dots,\,x_m.$$

    We say that $x,y \in X$ are in \emph{chain recurrence relation} (or simply in chain relation), and we write $x\,\mathcal{C}\,y$, if and only if, for every $\varepsilon>0$, there exists an $\varepsilon$-chain from $x$ to $y$.
   
\end{defi}

\begin{defi}\label{compchain}
    Let $(X,f)$ be a dynamical system.
    Assuming that, for $x,y\in X$, we have $x\,\mathcal{C}\,y$, we say that the sequence $\{C_n\}_{n\in\mathbb{N}}$ of $\varepsilon_n$-chains from $x$ to $y$ is a \emph{complete sequence of chains} if $\varepsilon_n$ converges to 0 monotonically.
\end{defi}

In \cite{Nested}, it was shown that on compact systems, the chain relation already encodes a much finer canonical structure.

\begin{defi}\label{nested}
  Let $(X,f)$ be a dynamical system. For $x,y\in X$, we say that $x\,\mathcal{C}_\subseteq\, y$ if and only if, for a sequence $\{\varepsilon_n\}_{n\in\mathbb{N}}$ of positive real numbers converging to 0 monotonically, there is a complete sequence $\{C_n\}_{n\in\mathbb{N}}$ of $\varepsilon_n$-chains from $x$ to $y$ such that $\widehat{C_n}\subseteq \widehat{C_{n+1}}$ for every $n\in\mathbb{N}$.

    We will call $\{C_n\}_{n\in\mathbb{N}}$ a \textit{sequence of nested chains for the sequence $\{\varepsilon_n\}_{n=1}^\infty$}.

\end{defi}

\begin{defi}
    For $x,y\in X$ and $\varepsilon>0$ let $C:x_0,\,x_1,\dots,\,x_m$ be an $\varepsilon$-chain  from $x=x_0$ to $y=x_m$. We say that $C$ contains a \emph{cyclic sub-chain} if one of the following conditions holds:
    \begin{itemize}
        \item for some $i,j\in\{1,\dots,m-1\}$ such that $i\neq j$, we have $x_i=x_j$;
        \item for some $i\in\{1,\dots,m-1\}$, we have $x_i=x_0$ or $x_i=x_m$.
    \end{itemize}
    We say that $C$ is \emph{acyclic} if it does not contain any cyclic sub-chain.
\end{defi}

    It is easily seen that any $\varepsilon$-chain $C$ with a cyclic sub-chain can be modified (by suitably eliminating some points) to get an $\varepsilon$-chain $C'$ with no cyclic sub-chain and such that $\widehat{C'}\subseteq \widehat{C}$.  
\begin{defi}\label{preceq}
    Let $(X,f)$ be a dynamical system.
    Assume that two points $x,y\in X$ are such that $x\,\mathcal{C}_\subseteq\,y$. We say that $x\,\mathcal{C}_\preceq\,y$ if the sequence of nested chains $\{C_n\}_{n\in\mathbb{N}}$ can be chosen so as to make all its chains without cyclic sub-chains.    
\end{defi}

In~\cite[Theorems 3.6 and 5.2]{Nested}, it has been proved that the three relations $\mathcal{C}$, $\mathcal{C}_\subseteq$ and $\mathcal{C}_\preceq$ coincide in compact dynamical systems with continuous maps.

\begin{thm}\label{thm:summary-previous}
Let $(X,f)$ be a compact dynamical system with $f$ continuous. For $x,y\in X$ the following are equivalent:
\begin{enumerate}
    \item[(1)] $x\,\mathcal{C}\,y$, i.e. for every $\varepsilon>0$ there exists an $\varepsilon$-chain from $x$ to $y$;
    \item[(2)] $x\,\mathcal{C}_{\subseteq}\,y$, i.e.\ there exists a vanishing sequence $\varepsilon_n\downarrow 0$ and a family of $\varepsilon_n$–chains $C_n$ from $x$ to $y$ with \emph{nested supports} $\widehat{C_n}\subseteq \widehat{C_{n+1}}$;
    \item[(3)] $x\,\mathcal{C}_{\preceq}\,y$, i.e.\ one can choose the nested family in (2) to be \emph{acyclic}.
\end{enumerate}
Consequently,
\[
\mathcal{C}\;=\;\mathcal{C}_{\subseteq}\;=\;\mathcal{C}_{\preceq}.
\]

\end{thm}
\begin{rem}
    The properties in (2)-(3), as shown in~\cite{Nested} (see Remark~2.13, Remark~2.16 and Lemma~2.18), are independent both of the metric and of the particular choice of the vanishing sequence $\{\varepsilon_n\}_{n\in\mathbb{N}}$.
\end{rem}

\begin{defi}
    Let $(A,\leq_A)$ and $(B,\leq_B)$ be two posets. 
    An \emph{order-isomorphism} between $A$ and $B$ is a bijection $f : A \to B$ such that for every $x,y \in A$, $x \leq_A y$ if and only if $f(x) \leq_B f(y)$.
    We write $(A,\leq_A) \cong (B,\leq_B)$ if such a map exists. 
    This defines an equivalence relation on the class of all posets.

    The \emph{order-type} of a poset $(A,\leq_A)$ is the equivalence class
    \[
        \operatorname{otp}(A,\leq_A) \;=\; \{ (B,\leq_B) \mid (B,\leq_B) \cong (A,\leq_A) \}.
    \]
\end{defi}

\begin{rem}
    Of course, the collection $\operatorname{otp}(A,\leq_A)$ is in general not a set but a proper class in ZF. 
    One common way to avoid this size issue is to work inside a \emph{universe of sets} (for example, the cumulative hierarchy $V$ in von Neumann's construction (see, for instance, \cite[p.~63]{Jech})), and to take order-types with respect to the posets contained in that universe. 
\end{rem}

We denote by:
    \begin{enumerate}
        \item [i)] $K$ the finite linear order on $K$ elements;
        \item [ii)] $\omega$ the first infinite ordinal;
        \item [iii)] $\omega^*$ the reverse order-type of $\omega$, that is, the order-type of non-positive integers under the usual relation $\leq$;
        \item[iv)] $\zeta=\omega^*+\omega$ the order-type of the integer numbers $\mathbb Z$;
        \item[v)] $\eta$ the countable dense order-type without extrema; 
        \item[vi)] $\omega_1$ the first uncountable ordinal.
        
\end{enumerate}
Recall the concept of \emph{very discrete rank} (or \emph{Hausdorff rank}) of a linear order (see~\cite[Def.~5.20]{rosenstein}).
\begin{defi}
    We define inductively for each countable ordinal $\alpha<\omega_1$ a class $VD_\alpha$ as follows 
    \begin{itemize}
        \item $0,1\in VD_0$.
        \item Given a linear order $I$ of order-type $\omega$, $\omega^*$, $\zeta$ or $K$, and for each $i\in I$ a linear order $L_i\in \cup\{VD_\beta \mid \beta<\alpha\}$, then $\sum \{L_i \mid i\in I\}\in VD_\alpha$.
    \end{itemize}
    The set $VD=\bigcup_{\alpha<\omega_1}VD_\alpha$ is called the class of (countable) \emph{very discrete linear orders}. 
    Thus, automatically, given a linear order $L\in VD$, there would be a smallest ordinal $\alpha$, called the \emph{$VD$-rank of $L$} and denoted $\operatorname{rank}_{VD}(L)$, such that $L\in VD_\alpha$.
\end{defi}
For example, we find $\operatorname{rank}_{VD}(\omega)=1$, $\operatorname{rank}_{VD}(\omega^n)=n$ ($n\in\mathbb{N}$), and $\operatorname{rank}_{VD}(\omega^\omega)=\omega$.

\bigbreak

\begin{defi}
    Let $(X,f)$ be a dynamical system and take $x,y\in X$ with $x\,\mathcal{C}_\preceq\,y$. An \emph{order-compatible nested sequence} from $x$ to $y$ is a sequence $\{C_n\}_{n\in\mathbb{N}}$ of acyclic $\varepsilon_n$-chains ($\varepsilon_n\downarrow0$) with nested supports $\widehat{C_n}\subseteq\widehat{C_{n+1}}$ such that, for every $n\in\mathbb{N}$ and for every $z,w\in \widehat{C_n}\setminus\{x,y\}$, if $z$ appears no later than $w$ in $C_n$, namely
    \[
    \min\{i\in \text{dom}(C_n) \mid C_n(i)=z\}  \le \min\{j\in \text{dom}(C_n) \mid C_n(j)=w\} ,
    \]
    then $z$ appears no later than $w$ in $C_{n+1}$, too.
\end{defi}
In~\cite{Nested} it has been proven that, for any pair $(x,y)\in X^2$, the \emph{Emergent Order Spectrum}  $\Omega_f(x,y)$ can be introduced:

\begin{defi}[Emergent Order Spectrum]\label{def:EOS}  
Let $\{C_n\}_{n\in\mathbb{N}}$ be an order-compatible nested sequence from $x$ to $y$ and set
$$
C=\bigcup_{n\in\mathbb{N}} \widehat{C_n}\setminus\{x,y\}.
$$ 
Define a linear order $\le_\infty$ on $C$ by
\begin{align*}
    z\le_\infty w \;\Longleftrightarrow\; &z=w\ \text{ or $z$ appears no later than $w$ in $C_n$} \  \text{for every $n\in\mathbb{N}$}\\
    &\text{such that $z,w\in\widehat{C_n}$}.
\end{align*}
The \emph{Emergent Order Spectrum} (EOS) of the pair $(x,y)\in X^2$ is
\begin{align*}
    \Omega_f(x,y) = \bigl\{\operatorname{otp}\bigl(C,\le_\infty\bigr) \mid \{C_n\}_{n\in\mathbb{N}} \text{ is an order-compatible nested sequence} \ \text{from $x$ to $y$} \bigr\}.
\end{align*}
When there is no ambiguity, we write $\Omega(x,y)$ for $\Omega_f(x,y)$. 
\end{defi}

The EOS is a topological invariant (see~\cite[Theorem~6.8]{Nested}):
\begin{thm}\label{conj__}
Let $(X,f)$ and $(Y,g)$ be topologically conjugate via a homeomorphism $h:X\to Y$, and set $H=h\times h:X^2\to Y^2$. Then
\[
\Omega_f=\Omega_g\circ H 
\qquad\text{and}\qquad 
\Omega_g=\Omega_f\circ H^{-1}.
\]
Equivalently, for all $(x,y)\in X^2$,
\[
\Omega_f(x,y)=\Omega_g\bigl(h(x),h(y)\bigr).
\]
\end{thm}

For our aim, we recall two results from~\cite{Nested}, namely
Theorems 6.9 and 6.13.

\begin{thm}\label{finite-ordinal}
    Let $x,y\in X$. 
    Then, there is 
    a finite ordinal $K$ such that
    $K\in\Omega(x,y)$ if and only if $y\in\mathcal{O}^+(x)$.
\end{thm}
\begin{thm} \label{lemma_transitiva}
    If $(X,f)$ is a transitive dynamical system with $\operatorname{card}(X)=\infty$, then, for every $x,y\in X$, we have $\eta\in\Omega(x,y)$.
\end{thm}

\bigskip

In this paper, we investigate how large the EOS can be under natural dynamical hypotheses.
Our main result, Theorem~\ref{long_t}, shows that for a transitive homeomorphism on a compact metric space without isolated points of cardinality $\mathfrak{c}$, the global spectrum $\Omega(X^2)$ contains every countable scattered order-type together with the countable dense order, and for a comeagre set of pairs $(x,y)$ the individual spectrum $\Omega(x,y)$ already realizes all countably infinite scattered types. 
Thus, chain recurrence, when combined with topological transitivity, encodes the entire Hausdorff hierarchy of countable scattered orders inside a single conjugacy invariant.
The proof combines set-theoretic $\alpha$-structures on orbit classes with Hausdorff’s analysis of scattered linear orders, and uses them to organize families of nested $\varepsilon_n$-chains according to their $VD$-rank and to assemble them via sums indexed by $K$, $\omega$, $\omega^*$, or $\zeta$.

We emphasize that our arguments are self-contained with respect to the EOS: 
nowhere do we assume a priori an existence theory for the Emergent Order Spectrum in the transitive setting.
Instead, the proof of our main result is constructive, and, in the process of realizing the desired order-types, it simultaneously yields the existence of a sequence of order-compatible nested $\varepsilon_n$-chains for transitive homeomorphisms. 
In this sense, Theorem~\ref{long_t} can be viewed as an independent existence result for the EOS under our hypotheses.

\section{Results}

Let us recall that a \emph{scattered} linear order is a linear order that does not contain a subset order-isomorphic to the rational numbers $\mathbb{Q}$.
A classical representation theorem by Hausdorff provides a general expression of linearly ordered sets in terms of scattered sets (see \cite[Theorem~4.9]{rosenstein}):

\begin{thm}[Hausdorff]
Any linear order \( L \) is a dense sum of scattered linear orders; that is, there is a dense linear order \( L_* \) and a map \( h \) from \( L_* \) to scattered linear orders such that $$ L = \sum \{ h(i) \mid i \in L_* \}. $$
\end{thm}
In turn, scattered sets admit a complete classification, also due to Hausdorff (see \cite[Theorem 5.24]{rosenstein}):
\begin{thm}[Hausdorff]
    The set of countable scattered linear orders is the smallest set of orders that contains singleton orders and is closed under concatenation indexed by K, $\omega$, or $\omega^*$.
\end{thm}

Let us introduce the concept of $\alpha$-structure and prove a set-theoretic technical result (Lemma~\ref{lem_structure}) that we will use in the proof of Theorem~\ref{long_t}.
\begin{defi}
Let $M$ be a set of cardinality at least $\aleph_1$.
For $\alpha\ge2$ a countable ordinal, we define recursively on $\alpha$ an $\alpha$-\emph{structure} $A_\alpha$ in $M$ (a family of families nested to depth $\alpha$) and its support $\operatorname{supp}(A_\alpha)$ as:
\begin{itemize}
  \item $\alpha = 2$: $A_2$ is a countably infinite subset of $M$ and $\operatorname{supp}(A_2)=A_2$.
  \item $\alpha =\beta +1$: $A_\alpha = \{B_i \mid i < \omega\}$, where each $B_i$ is a $\beta$-structure such that the $\operatorname{supp}(B_i)$ are pairwise disjoint and we set
  $$\operatorname{supp}(A_\alpha)=\bigcup_{i<\omega}\operatorname{supp}(B_i).$$
   \item $\alpha$ limit ordinal: Let $\{2\le\gamma_i < \alpha \mid i < \omega\}$ be an increasing sequence cofinal in $\alpha$. 
   We define $A_\alpha$ as $A_\alpha = \{B_i \mid i < \omega\}$, where each $B_i$ is a $\gamma_i$-structure such that the $\operatorname{supp}(B_i)$ are pairwise disjoint and we set
   $$
   \operatorname{supp}(A_\alpha)=\bigcup_{i<\omega}\operatorname{supp}(B_i).$$
\end{itemize}
\end{defi}


\begin{defi}\label{differenza}
    Given an $\alpha$-structure $A$ and a finite set $F$, we define the $\alpha$-structure $A\,\dot-\,F$ inductively over $\alpha$.
    \begin{itemize}
        \item If $\alpha=2$, then $A\,\dot-\,F=\text{supp}(A)\setminus F$;
        \item if $\alpha>2$, let $A=\{B_i\mid i<\omega,\text{ where }B_i\text{ are }\beta_i\text{-structure with }\beta_i<\alpha\}$. Then, we set $$A\,\dot-\,F=\{B_i\,\dot-\,F\mid i<\omega\}.$$
    \end{itemize}
\end{defi}
\begin{rem}
    Given an $\alpha$-structure $A$ and a finite set $F$, the $\alpha$-structure $A\,\dot-\,F$ is well defined, because if $\alpha=2$, supp$(A)\setminus F$ is an infinite set.
\end{rem}

\begin{lem}\label{lem_structure}
Let $M$ be a set with $|M| \geq \aleph_1$. There exists a family
$\{A_\alpha \mid \alpha < \omega_1\}$ of structures, where each $A_\alpha$ is an $\alpha$-structure in $M$ and their supports are pairwise disjoint. 
\end{lem}

\begin{proof}
Fix a subset $U \subseteq M$ with $|U| = \aleph_1$ and let $U=\bigsqcup_{\alpha<\omega_1}T_\alpha$ be a partition of $U$ in $\aleph_1$ sets of cardinality $\aleph_1$.
We construct the family $\{A_\alpha \mid \alpha < \omega_1\}$ recursively (by transfinite induction on $\alpha$), ensuring each $A_\alpha$ is an $\alpha$-structure with support contained in $T_\alpha$:
\begin{itemize}
   \item $\alpha = 2$: Set $A_2$ to be a countably infinite subset of $T_2$. Its support is $A_2$.
\end{itemize}
Let now $\alpha>2$ and assume that, for every $\lambda<\alpha$ and for every set $V\subseteq M$ with $|V|=\aleph_1$, we can construct a $\lambda$-structure supported in $V$. 
Let us treat the case $\alpha$ successor ordinal and the case $\alpha$ limit ordinal separately.
\begin{itemize}  
    \item $\alpha = \beta + 1$:  Partition $T_\alpha = \bigsqcup_{i < \omega} T_{\alpha,i}$, where each $T_{\alpha,i}$ is uncountable. 
    For each $i < \omega$, by the induction hypothesis applied to $V = T_{\alpha,i}$, construct a $\beta$-structure $B_i$ with support contained in $T_{\alpha,i}$.
    \\
    Set $A_\alpha = \{B_i \mid i < \omega\}$. Since the $T_{\alpha,i}$ are pairwise disjoint, the supports of the $B_i$ are pairwise disjoint.
    
    \item $\alpha$ limit ordinal: Choose an increasing sequence $\{2\le\gamma_i < \alpha \mid i < \omega\}$ of ordinals cofinal in $\alpha$. Partition $T_\alpha = \bigsqcup_{i < \omega} T_{\alpha,i}$, where each $T_{\alpha,i}$ is uncountable. 
    For every $i < \omega$, construct a $\gamma_i$-structure $B_i$ with support contained in $T_{\alpha,i}$, using the induction hypothesis on $V = T_{\alpha,i}$.
    \\
    Set $A_\alpha = \{B_i \mid i < \omega\}$. 
   The supports of the $B_i$ are pairwise disjoint since the $T_{\alpha,i}$ are.

\end{itemize}
The $A_\alpha$ have pairwise disjoint supports because their supports are contained in the disjoint sets $T_\alpha$.

Thus, the family $\{A_\alpha \mid 2\le\alpha < \omega_1\}$ satisfies the claim.

\end{proof}


\noindent We are now ready to prove our main result. 
The proof proceeds in two main stages. 
First, for a generic pair $(x,y)\in X^2$ with disjoint two–sided orbits, transitivity and the density of forward and backward orbits allow us to construct order-compatible nested families of $\varepsilon_n$-chains that realize the basic building blocks of countably infinite scattered orders, namely order-types of the form $\omega+k$, $h+\omega^*$, and $h+\zeta+k$. 
This is done in Lemma \ref{lem_partial}.
Then, in the main Theorem, we organize orbit classes into suitable ``$\alpha$-structures'' indexed by countable ordinals $\alpha<\omega_1$, and use these to carry out a transfinite induction on the Hausdorff rank: at stage $\alpha$ we assign disjoint orbit classes to the scattered orders of rank$_{VD}\le\alpha$ and concatenate the corresponding chains along sums indexed by $K$, $\omega$, $\omega^*$ or $\zeta$, thereby realizing arbitrary countable scattered order-types in $\Omega(x,y)$ for comeagre many pairs $(x,y)$.

\begin{lem}\label{lem_partial}
    Let $(X,f)$ be a compact dynamical system such that $X$ has no isolated points, $|X|=\mathfrak{c}$, and $f$ is a transitive homeomorphism of $X$. Let $\mathcal{I}=\{x\in X\,|\,\overline{\mathcal{O}^-(x)}=\overline{\mathcal{O}^+(x)}=X\}$ be the set of points with both a dense backward orbit and a dense forward orbit. Set
    \begin{equation*}\label{gamma}
        \Gamma:=\left\{\omega+k,\,h+\omega^*,\,h+\zeta+k\,|\,k,h\in\mathbb N_0\right\}.
    \end{equation*}
    If $x,y\in \cal I$ are such that $\mathcal{O}(x)\cap \mathcal{O}(y)=\varnothing$,
    then
    $$
    \Gamma\subseteq\Omega(x,y). 
    $$ 
    More precisely, for every $\xi\in \Gamma$, there exists a sequence of order-compatible nested chains $\{C_n\}_{n\in\mathbb{N}}$ such that, setting $C=\big(\bigcup_n\widehat{C_n}\big)\setminus\{x,y\}$, the order-type of $(C,\le_\infty)$ is $\xi$.
\end{lem}
\begin{proof}
    Fix a sequence of positive real numbers $\{\varepsilon_n\}_{n\in\mathbb N}$ that converges to zero monotonically.    
    Let us distinguish three different cases.
    \begin{enumerate}
        \item Let $\xi=\omega+k$, for some $k\in\mathbb N_0$. 
        Since $x\in\mathcal{I}$, it is possible to find a strictly increasing sequence of natural numbers $\{h_n\}_{n\in\mathbb{N}}$ such that
        \begin{equation}\label{dist_1}
            d(f^{h_n+1}(x),f^{-k}(y))<\varepsilon_n\qquad \forall n\in\mathbb{N}.
        \end{equation}
        For every $n\in\mathbb{N},$ we set 
        $$
        C_n:\,x,\,f(x),\dots,\,f^{h_n-1}(x),\,f^{h_n}(x),\,f^{-k}(y),\,f^{-k+1}(y),\dots,\,f^{-1}(y),\,y.
        $$
        By \eqref{dist_1}, $C_n$ is an $\varepsilon_n$-chain and, by construction, $\widehat{C_n}\subseteq\widehat{C_{n+1}}$ for every $n\in\mathbb{N}$. 
        Since $\mathcal{O}(x)\cap\mathcal{O}(y)=\varnothing$, it follows that $\{C_n\}_{n\in\mathbb{N}}$ is  a sequence of nested and acyclic chains from $x$ to $y$. Moreover, since, by construction, the relative order in which two points appear in the chain $C_n$ does not change in the chain $C_{n+1}$, the sequence $\{C_n\}_{n\in\mathbb N}$ is a sequence of order-compatible nested chains.
        Set $C=\big(\bigcup_n\widehat{C_n}\big)\setminus\{x,y\}$ and notice that 
        $$
        C=\mathcal{O}^+(x)\cup \{f^{-k}(y),f^{-k+1}(y),\ldots,f^{-1}(y)\}.
        $$
        Therefore, the application $\varphi:(C,\le_\infty)\longrightarrow\omega+k$ given by
        \begin{equation*}
        \varphi(u)=
            \begin{cases}
                \ell &\text{if }u=f^\ell(x)\text{ for some }\ell\in\mathbb N\\
                \omega+k-s &\text{if }u=f^{-s}(y)\text{ for some } s\in\{1,2,\dots,k\}
            \end{cases}
        \end{equation*}
        is an order-isomorphism between $(C,\le_\infty)$ and $\omega+k$ and so the order-type of $(C,\le_\infty)$ is $\xi$.

        \item Let $\xi=h+\omega^*$ for some $h\in\mathbb N_0$. 
        Since $\mathcal{O}^-(y)$ is dense in $X$, there exists a strictly increasing sequence of natural numbers $\{k_n\}_{n\in\mathbb{N}}$ such that
        \begin{equation}\label{dist_2}
            d(f^{h+1}(x),f^{-k_n}(y))<\varepsilon_n\qquad \forall n\in\mathbb{N}.
        \end{equation}
        For every $n\in\mathbb{N},$ we set
        $$
        C_n:\,x,\,f(x),\dots,\,f^{h-1}(x),\,f^h(x),\,f^{-k_n}(y),\,f^{-k_n+1}(y),\dots,\,f^{-1}(y),\,y.
        $$
        By \eqref{dist_2}, $C_n$ is an $\varepsilon_n$-chain and, by construction, $\widehat{C_n}\subseteq\widehat{C_{n+1}}$ for every $n\in\mathbb{N}$.
        Since $\mathcal{O}(x)\cap\mathcal{O}(y)=\varnothing$, we have that $\{C_n\}_{n\in\mathbb{N}}$ is a sequence of nested and acyclic chains from $x$ to $y$.  Moreover, since, by construction, the relative order in which two points appear in the chain $C_n$ does not change in the chain $C_{n+1}$, the sequence $\{C_n\}_{n\in\mathbb N}$ is a sequence of order-compatible nested chains.
        Set $C=\big(\bigcup_n\widehat{C_n} \big)\setminus\{x,y\}$ and notice that 
        $$
        C=\{f(x),\ldots,f^h(x)\}\cup \mathcal{O}^-(y).
        $$
        Enumerating $h+\omega^*$ according to its inner order as
         $$
         h+\omega^*=\{a_1,\,a_2,\dots,\,a_h\}\cup\{c_{-\ell}\}_{\ell\in\mathbb N}=\{a_1,\,a_2,\dots,\,a_h,\dots,c_{-\ell},\dots,\,c_{-2},\,c_{-1}\},
         $$
        we find that the application $\varphi:(C,\le_\infty)\longrightarrow h+\omega^*$ given by
        \begin{equation*}
        \varphi(u)=
            \begin{cases}
                a_\ell &\text{if }u=f^\ell(x)\text{ for some }\ell\in \{1,2,\dots,h\}\\
                c_{-\ell}&\text{if }u=f^{-\ell}(y)\text{ for some }\ell\in \mathbb{N}
            \end{cases}
        \end{equation*}        
         is an order-isomorphism between $(C,\le_\infty)$ and $h+\omega^*$ and so the order-type of $(C,\le_\infty)$ is $\xi$.

        \item Let $\xi=h+\zeta+k$, for some $h,k\in\mathbb N_0$. 
        Since, by Remark~\ref{rem_trans}, $\mathcal{I}$ is a comeagre set, there exists $z\in\mathcal{I}$ such that $z\notin\mathcal{O}(x)\cap\mathcal{O}(y)$.
        Since $f$ is a homeomorphism, it follows that $\mathcal{O}(z)\cap\mathcal{O}(x)=\mathcal{O}(z)\cap\mathcal{O}(y)=\varnothing$. Then, it is possible to consider two strictly increasing sequences of natural numbers $\{k_n\}_{n\in\mathbb{N}}$ and $\{h_n\}_{n\in\mathbb{N}}$ such that, for every $n\in\mathbb{N}$,
        \begin{equation}\label{dist_3}
        d(f^{h+1}(x),f^{-k_n}(z))<\varepsilon_n\quad \text{and} \quad d(f^{h_n+1}(z),f^{-k}(y))<\varepsilon_n.
         \end{equation}
         
        For every $n\in\mathbb{N},$ we set
        \[
            \begin{aligned}
                C_n:\,& x,\,f(x),\,\dots,\,f^{h-1}(x),\,f^h(x),\,f^{-k_n}(z),\,f^{-k_n+1}(z),\,\dots,\,f^{-1}(z),\\
                &z,\,f(z),\,\dots,\,f^{h_n-1}(z),\,f^{h_n}(z),\,f^{-k}(y),\,f^{-k+1}(y),\,\dots,\,f^{-1}(y),\,y.
            \end{aligned}
        \]
        By \eqref{dist_3}, $C_n$ is an $\varepsilon_n$-chain and, by construction, $\widehat{C_n}\subseteq\widehat{C_{n+1}}$ for every $n\in\mathbb{N}$.
        Since $\mathcal{O}(x)\cap\mathcal{O}(y)=\varnothing$, we have that $\{C_n\}_{n\in\mathbb{N}}$ is a sequence of nested and acyclic chains from $x$ to $y$.  Moreover, since, by construction, the relative order in which two points appear in the chain $C_n$ does not change in the chain $C_{n+1}$, the sequence $\{C_n\}_{n\in\mathbb N}$ is a sequence of order-compatible nested chains.
        Set $C=\big(\bigcup_n\widehat{C_n}\big) \setminus\{x,y\}$ and notice that 
        $$
        C=\{f(x),\ldots,f^h(x)\}\cup \mathcal{O}(z)\cup \{f^{-k}(y),\ldots,f^{-1}(y)\}.
        $$
        Enumerating $h+\zeta+k$ according to its inner order as
        \begin{align*}
            h+\zeta+k&=\{a_1,\,a_2,\dots,\,a_h\}\cup\{b_\ell\}_{\ell\in\mathbb Z}\cup\{c_1,\,c_2,\dots,c_k\}=\\
            &=\{a_1,\,a_2,\dots,\,a_h,\dots,b_{-\ell},\dots,\,b_{-1},\,b_0,\,b_1,\dots,\,b_\ell,\dots,\,c_1,\,c_2,\dots,c_k\},
        \end{align*}
        we find that the application $\varphi:(C,\le_\infty)\longrightarrow h+\zeta+k$ given by
        \begin{equation*}
        \varphi(u)=
            \begin{cases}
                a_\ell &\text{if }u=f^\ell(x)\text{ for some }\ell\in \{1,2,\dots,h\}\\
                b_\ell &\text{if }u=f^{\ell}(z) \text{ for some } \ell\in \mathbb{Z}\\
                c_{k-\ell+1}&\text{if }u=f^{-\ell}(y)\text{ for some }\ell\in \{1,2,\dots,k\}
            \end{cases}
        \end{equation*}
is an order-isomorphism between $(C,\le_\infty)$ and $h+\zeta+k$ and so the order-type of $(C,\le_\infty)$ is $\xi$.

\end{enumerate}

\end{proof}

\begin{thm}\label{long_t}
    Let $(X,f)$ be a compact dynamical system such that $X$ has no isolated points and $|X|=\mathfrak{c}$. 
    If $f$ is a transitive homeomorphism, then every countable scattered order and the order-type of the rationals $\eta$ appear in $\Omega(X^2)$.
    
\vspace{0.3cm}

    More precisely:
    \begin{itemize}
        \item[1.]There exists a comeagre set $\mathcal{M}\subseteq X^2$ such that, for every $(x,y)\in \mathcal{M}$, the family of orders $\Omega(x,y)$ contains every countably infinite scattered order;
        \item[2.] $\Omega(x,y)$ contains the finite ordinal $K$ if and only if $y\in\mathcal{O}^+(x)$;
        \item[3.]$\Omega(x,y)$ contains the order-type $\eta$ for every $x,y\in X$. 
    \end{itemize}
\end{thm}
\begin{proof}
The point 2. and the point 3. are satisfied by, respectively, Theorem \ref{finite-ordinal} and Theorem \ref{lemma_transitiva}. Now we prove the first point.
    Let 
    $$
    \mathcal{I}=\{x\in X\,|\,\overline{\mathcal{O}^-(x)}=\overline{\mathcal{O}^+(x)}=X\}
    $$ 
    be the set of points with both a dense backward orbit and a dense forward orbit. It is known that there is a comeagre set
    $$
    \mathcal{M}=\{(x,y)\in\mathcal{I}^2\,|\,\mathcal{O}(x)\cap\mathcal{O}(y)=\varnothing\}
    $$
    (see, for instance, \cite{DFR}).
    We can define the following equivalence relation over $\mathcal{I}$: for every $x,y\in\mathcal{I}$, we say $x\sim y$ if and only if $\mathcal{O}(x)\cap\mathcal{O}(y)\neq\varnothing$. Since $f$ is a bijection, it is not hard to see that $\sim$ is an equivalence relation and that the quotient set $\mathcal{I}_{/_\sim}$ has cardinality $\mathfrak{c}$. We denote by $[x]$ the equivalence class of $x\in\mathcal{I}$. 
    By Lemma \ref{lem_structure}, there exists a family of $\alpha$-structures ($\alpha<\omega_1$) in $\mathcal{I}_{/_\sim}$ with pairwise disjoint supports.

    Now we proceed by transfinite induction on the rank of scattered orders appearing in $\Omega(X^2)$.
    More precisely, we prove, by transfinite induction on the countable ordinal $\alpha\ge 2$, that, for every $(p,q)\in\mathcal{M}$, given an arbitrary $\alpha$-structure $A$ in $\mathcal{I}_{/_\sim}$, for any countably infinite scattered order $\xi$ of $VD$-rank $\beta\le\alpha$, we can find a family of ordinately nested chains $\{C_n\}_{n\in\mathbb{N}}$ from $p$ to $q$ such that, setting $C=\big(\bigcup_n \widehat{C_n}\big) \setminus\{p,q\}$ and $C_n: \,x_0=p,\,x_1,\dots,\,x_{m-1},\,x_m=q$,
    \begin{itemize}
        \item[A)] the ordered limit set $(C,\le_\infty)$ has order-type $\xi$;
        \item[B)] $C$ consists of points $z$ with classes $[z]$ in $\operatorname{supp}(A)\cup\{[p],\,[q]\}$;
        \item[C)] if $x_j\in[p]$ for some $j<m$, then $x_{h+1}=f(x_h)\in [p]$ for every $0\le h< j$;
        \item[D)] if $x_j\in[q]$ for some $j<m$, then $x_h=f(x_{h-1})\in [q]$ for every $j<h\le m$.
    \end{itemize}

    \medskip\textbf{STEP 1.}
    Let $\alpha=2$, fix a $2$-structure $A$ in $\mathcal{I}_{/\sim}$, that is a countably infinite subset of $\mathcal{I}_{/_\sim}$, take $(p,q)\in\mathcal{M}$ and let $\xi$ be a countably infinite scattered order with rank$_{VD}(\xi)\le2$. 
    If rank$_{VD}(\xi)=0$, then $\xi\in VD_0=\{0,1\}$, and so we have nothing to prove because $\xi$ is finite.
    If rank$_{VD}(\xi)=1$, then $\xi$ is of the form $K,\,\omega,\,\omega^*$ or $\zeta$. 
    In particular, we do not have to consider the case $\xi=K$ because it is a finite order. 
    In the other cases, by using the same construction as in the proof of Lemma \ref{lem_partial}, with $h,k=0$, we have $\{\omega,\,\omega^*,\,\zeta\}\subseteq \Omega(p,q)$. 
    Moreover, if $\xi=\omega$ or $\xi=\omega^*$, then the set $C$ consists of points in $[p]$ or $[q]$, respectively, while if $\xi=\zeta$, then the set $C$ also consists of points in $[z]$, where $[z] \in \text{supp}(A\,\dot-\,\{[p],[q]\})$. In this case, the construction found in the proof of Lemma \ref{lem_partial} assures us that conditions A), B), C), and D) hold.
    
   Now, let us consider the case where rank$_{VD}(\xi)=2$, which means that 
   $$
   \xi=\sum_{\ell\in I} \xi_\ell
   $$ 
   where $I\in \{\omega,\omega^*,\zeta,K\}$ and rank$_{VD}(\xi_\ell)<2$ for every $\ell\in I $. 
   We can equivalently write $\xi$ as 
   $$
   \xi=\sum_{\ell \in I'} h_\ell+\xi_\ell+k_\ell=\sum_{\ell \in I'} \xi'_\ell 
   $$ 
   where $\xi_\ell\in\{\omega, \omega^*, \zeta\}$, $\xi'_\ell=h_\ell +\xi_\ell+k_\ell$ with $h_\ell,k_\ell \in \mathbb{N}_0$ for every $\ell\in I'$, and $I'\in \{\omega,\omega^*,\zeta,K'\}$.
    
    For simplicity, we consider only the case of the sum indexed by $\omega^*$: the case of the sum indexed by $\omega$ can be proved with a specular argument, the case of the finite sum can be studied with minor changes, and the case of the sum indexed by $\zeta$ can be obtained as the sum of two sums, the former indexed by $\omega^*$ and the latter indexed by $\omega$.
    
    With an abuse of notation, we write $\omega^*=\{\dots,\,-n,\dots,\,-2,\,-1\}$. 
    We first show how to construct special sequences of order-compatible nested chains, denoted by $\{S_n^\ell\}_{n\in\mathbb{N}}$, that will be used in the following to define a sequence of order-compatible nested $1/n$-chains $\{C_n\}_{n\in\mathbb{N}}$ such that, setting 
    $$
    C=\big(\bigcup_n\widehat{C_n}\big) \setminus\{x,y\},
    $$ 
    the ordered limit set $(C,\le_\infty)$ has order-type
    $$
    \operatorname{otp}(C,\le_\infty)=\sum_{\ell\in\omega^*}h_\ell+\xi_{\ell}+k_\ell=\sum_{\ell\in\omega^*}\xi_\ell',
    $$
    where, for every $\ell\in\omega^*$, $\xi_\ell\in\{\omega,\,\omega^*,\,\zeta\}$ and $\xi_\ell'=h_\ell+\xi_\ell+k_\ell$ with $h_\ell,k_\ell\in\mathbb N_0$. 
    Note that, if $\xi_\ell$ is $\omega$ (or $\omega^*$), then we can assume $h_\ell$ (respectively $k_\ell$) to be 0.

    Let $B_1\sqcup B_2$ be a partition of supp$(A\,\dot-\,\{[p],\,[q]\})$ such that both $B_1$ and $B_2$ are infinite.
    Up to re-numbering the classes in $B_1$ and in $B_2$, we can assume
    $$
    B_1=\{[v_{-1}],\,[v_{-2}],\dots,\,[v_{\ell}],\dots\},
    $$
    and
    $$
    B_2=\{[w_{-1}],\,[w_{-2}],\dots,\,[w_{\ell}],\dots\}.
    $$
    For every $\ell\in\omega^*=\{\dots,\,-n,\dots,\,-2,\,-1\}$, there exists a point $z_{\ell}\in[v_{\ell}]=\mathcal{O}(v_{\ell})$ such that
    \begin{equation}\label{eq_z_n}
        d(f(p),z_{\ell})<-\frac{1}{2\ell}.
    \end{equation}    
    
    Set $z_0=q$.
    For every $\ell \in\mathbb \omega^*$, let us now define a sequence of order-compatible nested $\frac{1}{2n}$-chains $\{S^\ell_n\}_{n\in\mathbb{N}}$.

    Since $(f^{-1}(z_\ell),z_{\ell+1})\in \mathcal{M}$, performing the same construction as in the proof of Lemma \ref{lem_partial} with $\{\varepsilon_n\}_{n\in\mathbb N}=\{\frac{1}{2n}\}_{n\in\mathbb N}$, we obtain a sequence of order-compatible nested $\frac{1}{2n}$-chains $\{R^\ell_n\}_{n\in\mathbb{N}}$ from $f^{-1}(z_\ell)$ to $z_{\ell+1}$ such that, setting 
    $$
    R^\ell=\big(\bigcup_n\widehat{R^\ell_n}\big)\setminus\{f^{-1}(z_\ell),z_{\ell+1}\},
    $$
    it satisfies
    \begin{equation*}
        \operatorname{otp}(R^\ell,\le_\infty)=\xi_\ell'.
    \end{equation*}
    Then, we define the sequence of nested $\frac{1}{2n}$-chains $\{S_n^\ell\}_{n\in\mathbb{N}}$, where each $S^\ell_n$ is obtained by deleting both the first and the last points of the chain $R_n^\ell$, that is, by setting 
    $$
    S_n^\ell={R^\ell_n}_{|\operatorname{dom}(R_n^\ell)\setminus\{0,\max(\operatorname{dom}(R_n^\ell))\}}.
    $$
    Notice that 
    $$
    \bigcup_n\widehat{S^\ell_n}=R^\ell.
    $$ 
    This step of removing the endpoints from the already-obtained sequences of nested chains $\{R^\ell_n\}_{n\in\mathbb{N}}$ is necessary to ``glue'' them together in the desired manner in the following. 
   
    For every $n\in\mathbb N$, setting $u^\ell_n$ and $v^\ell_n$ to be respectively the first and last points of the chain $S^\ell_n$, that is, $u^\ell_n=S^\ell_n(0)$ and $v^\ell_n=S^\ell_n(\max(\operatorname{dom}(S^\ell_n))$, we have
    \begin{equation}\label{cond_primo_ultimo}
        d(z_\ell,u^\ell_n)<\frac{1}{2n}\quad\text{and}\quad d(z_{\ell+1},f(v^\ell_n))<\frac{1}{2n},
    \end{equation}
    because $R_n^\ell$ is a $\frac{1}{2n}$-chain.

    Notice that we can define the sequence of chains $\{S^\ell_n\}_{n\in\mathbb N}$ by choosing points whose classes are in supp$(A\,\dot-\,[p])$. 
    More precisely, whenever we must build an order $\xi_\ell'$ of the form $\omega+k$ or $h+\omega^*$, it is enough to consider the points in $[z_\ell]$ and in $[z_{\ell+1}]$. On the other hand, if $\xi_\ell'$ contains $\zeta$,
    we perform the construction in the proof of Lemma \ref{lem_partial} taking $z=w_\ell$ (remind that $[w_\ell]\in B_2$ and so $[w_\ell]\notin\{[z_\ell],[z_{\ell+1}]\}$).

    Furthermore, by construction, the sequences of chains $\{S^\ell_n\}_{n\in\mathbb{N}}^{\ell\in\omega^*}$ satisfy
    \begin{equation}\label{disg_aciclicità}
        \left(\,\bigcup_n\widehat{S^{\ell_1}_n}\,\right)\cap\left(\,\bigcup_{n}\widehat{S^{\ell_2}_n}\,\right)=\varnothing,\quad\text{for every }\ell_1\neq \ell_2.
    \end{equation}


    In order to obtain a sequence of order-compatible nested $1/n$-chains $\{C_n\}_{n\in\mathbb{N}}$ from $p$ to $q$, for every $n\in\mathbb{N}$ we first define the chain
    $$
    T_n=\,S^{-n}_n\,\dot\sqcup\,S^{-n+1}_n\,\dot\sqcup\dots\dot\sqcup\,S^{-2}_n\,\dot\sqcup\,S^{-1}_n,
    $$
    where, given two chains $A:\,a_0,\,a_1,\dots,\,a_{m_1}$ and $B:\,b_0,\,b_1,\dots,\,b_{m_2}$, by $A\,\dot\sqcup\,B$ we mean the chain
    $$
    A\,\dot\sqcup\,B:\,a_0,\,a_1,\dots,\,a_{m_1},\,b_0, \,b_1,\dots,\,b_{m_2}.
    $$
    Notice that, in general, the chain $T_n$ does not go from $p$ to $q$.
    
    \medskip
    \textbf{STEP 2.} Let us prove now that $T_n:\,x_0,\,x_1,\dots,\,x_m$ is a $1/n$-chain, for every $n\in\mathbb N$. 
    We need to check that
    \begin{equation}\label{dis-chain}
        d(f(x_i),x_{i+1})<\frac{1}{n}
    \end{equation}
    for every $i=0,1,\dots,m$.
    
    If $x_i$ and $x_{i+1}$ belong to the same chain $S^\ell_n$ for some $\ell\in\{-n,\dots,-1\}$, since $S^\ell_n$ is a $\frac{1}{2n}$-chain, then inequality \eqref{dis-chain} trivially holds.
    
    If $x_i\in\widehat{S^{\ell-1}_n}$ and $x_{i+1}\in\widehat{S^\ell_n}$, for some $\ell\in\{-n+1,\dots,-1\}$, then $x_i=v^{\ell-1}_n$ and $x_{i+1}=u^{\ell}_n$. By \eqref{cond_primo_ultimo}, we have 
    $$
    d(f(x_i),x_{i+1})=d(f(v^{\ell-1}_n),u^{\ell}_n)\le d(f(v^{\ell-1}_n),z_{\ell})+d(z_{\ell},u^{\ell}_n)<\frac{1}{n}.
    $$
    Therefore, $T_n$ is a $1/n$-chain. 
    However, $T_n$ is not a chain from $p$ to $q$.
    Then, for every $n\in\mathbb{N}$, we define a $1/n$-chain from $p$ to $q$ as
    $$
    C_n:\,p\,\dot\sqcup\,T_n\,\dot\sqcup\,q.
    $$
    Indeed, the first point of $T_n$ is $u^{-n}_n$, which satisfies $d(z_{-n},u^{-n}_n)<\frac{1}{2n}$. 
    Moreover, by \eqref{eq_z_n} we have  $d(f(p),z_{-n})<\frac{1}{2n}$.
    Therefore,
    $$
    d(f(p),u^{-n}_n)\le d(f(p),z_{-n}) + d(z_{-n},u^{-n}_n)<\frac{1}{n}.
    $$ 
    Similarly, the last point of $T_n$ is $v^{-1}_n$ and by \eqref{cond_primo_ultimo} we have
    $$
    d(f(v^{-1}_n),q)=d(f(v^{-1}_n),z_0)<\frac{1}{2n}<\frac{1}{n}.
    $$

    Notice that, for every $n\in\mathbb{N}$, we have $\widehat{C_n}\subseteq \widehat{C_{n+1}}$, because $\widehat{T_n}\subseteq\widehat{T_{n+1}}$. Moreover, \eqref{disg_aciclicità} guarantees that the sequence $\{C_n\}_{n\in\mathbb N}$ is a sequence of acyclic nested chains. 
    Furthermore, the relative order in which two points appear in the chain $C_n$ does not change in the chain $C_{n+1}$ for every $n\in\mathbb{N}$.
    Then, $\{C_n\}_{n\in\mathbb N}$ is a sequence of order-compatible nested chains. 
    
    By construction, setting $C=\big(\bigcup_n \widehat{C_n}\big)\setminus\{p,q\}$, we have that
    \begin{equation*}
        \text{otp}(C,\le_\infty)=\text{otp}(\cup_n \widehat{T_n},\le_\infty)=\sum_{\ell\in\omega^*}\xi'_\ell=\xi.
    \end{equation*}
    Thus, condition A) holds. 
    Since we used only points whose classes lie in $$B_1\cup B_2\cup\{[p],\,[q]\},$$
    condition B) holds as well.
    Moreover, by the definition of $T_n$, since $\bigcup_n \widehat{S}^\ell_n \cap [p]=\varnothing$ for every $\ell\in \omega^*$, it follows that 
    $$
    \widehat{T_n}\cap[p]=\varnothing\quad\forall n\in\mathbb N.
    $$
    Then, $\big(\bigcup_n \widehat{C_n}\big) \cap[p]=\{p\}$ and so, condition C) holds. 
    Finally, by construction, for every $\ell<-1$ and for every $n\in\mathbb N$, since $\widehat{S_n^\ell}\subseteq[z_\ell]\cup[z_{\ell+1}]\cup[w_\ell]$ and since $([z_\ell]\cup[z_{\ell+1}]\cup[w_\ell])\cap[q]=\varnothing$, it follows that $\widehat{S^\ell_n}\cap[q]=\varnothing$ and
    $$
    \widehat{S^{-1}_n}\cap[q]=\{f^{-k_{-1}}(q),\,f^{-k_{-1}+1}(q),\dots,\,f^{-1}(q),\,q\}.
    $$
    Hence, by the definition of $S_n^{-1}$, condition D) holds.


    Then, the points A), B), C), and D) are satisfied for $\alpha=2$.

    \medskip\textbf{STEP 3.} Let us proceed now by transfinite induction.
    Let $\alpha>2$ be a countable ordinal number.
    Assume that for every $\beta<\alpha$, properties A), B), C), and D) are satisfied. 
    More precisely, given an arbitrary $\beta$-structure $B$ in $\mathcal{I}_{/_\sim}$ and $(p,q)\in\mathcal{M}$, for any countably infinite scattered order $\xi$ with rank$_{VD}(\xi)=\gamma\le\beta$, we can find a sequence of order-compatible nested chains $\{C_n\}_{n\in\mathbb{N}}$ from $p$ to $q$ such that the ordered limit set $(C,\le_\infty)$ has order-type $\xi$ and $C=\big(\bigcup_n \widehat{C_n}\big)\setminus\{p,q\}$ consists of points $z$ with classes $[z]$ in $\operatorname{supp}(B)\cup\{[p],\,[q]\}$.
    Moreover, setting
    $$C_n: \,x_0=p,\,x_1,\dots,\,x_{m-1},\,x_m=q,$$
    if $x_j\in [p]$ for some $j<m$, then $x_{h+1}=f(x_{h})\in [p]$ for every $0\le h<j$; whereas if $x_j\in[q]$ for some $j<m$, then $x_{h}=f(x_{h-1})\in[q]$ for every $j<h\le m$.
    
    Fix an $\alpha$-structure $A=\{B_i\mid i<\omega\}$ in $\mathcal{I}_{/_\sim}$ and $(x,y)\in\mathcal{M}$.    
    Set
    $$A'=A\dot{-} \{[x],[y]\}=\{B_i'\mid i<\omega\},$$
    where $B_i'=B_i\dot{-}\{[x],[y]\}$, and let $\xi$ be a countable scattered order such that rank$_{VD}(\xi)=\gamma\le \alpha$. 
    If $\gamma<\alpha$, then there is a $\gamma$-structure $B$ within $A'$ and, by inductive hypothesis, we know that we can define a sequence of order-compatible nested chains $\{C_n\}_{n\in\mathbb{N}}$ that satisfies the claim.
    
    Assume that rank$_{VD}$$(\xi)=\alpha$, which means that
    $$
    \xi=\sum_{\ell\in I}\xi_\ell,
    $$
    where $I\in \{\omega,\,\omega^*,\,\zeta, K\}$ and rank$_{VD}$$(\xi_\ell)<\alpha$ for every $\ell\in I$. 
    As observed in \textbf{STEP 1}, we can equivalently write $\xi$ as
    $$
    \xi=\sum_{\ell\in I'}h_\ell+\xi_\ell+k_\ell=\sum_{\ell\in I'}\xi_\ell'
    $$
    where $\xi_\ell$ is a countable scattered order, $\xi'_\ell=h_\ell +\xi_\ell +k_\ell$ with $h_\ell,k_\ell\in \mathbb{N}_0$  and rank$_{VD}(\xi_\ell)<\alpha$ for every $\ell\in I'$, where $I'\in\{\omega,\,\omega^*,\,\zeta, K'\}$.
    As in \textbf{STEP 1}, we will show just the case $I'=\omega^*=\{\dots,\,-n,\dots,\,-2,\,-1\}$.
    
    \begin{itemize}
        \item If $\alpha=\beta+1$ for some ordinal number $\beta$, then the $\alpha$-structure $A'$ is a countable family of $\beta$-structures, that is
        $$
        A'=\{B_i'\mid i<\omega\},
        $$
        where $B_i'$ is a $\beta$-structure for every $i<\omega$.
        Up to considering a bijection between the set $\{i\mid i<\omega\}$ and the set $\omega^*=\{\dots,\,-n,\dots,\,-2,\,-1\}$, we can enumerate the $\beta$-structures forming $A'$ with $\omega^*$.
        For every $\ell\in\mathbb \omega^*$, take a class $[v_\ell]\in B_\ell'$. 
        As before, for every $\ell\in\omega^*$ there exists a point $z_\ell\in[v_\ell]$ such that
        \begin{equation}\label{eq_zl_ind}
            d(f(x),z_\ell)<-\frac{1}{2\ell}.
        \end{equation}
        Set $z_0=y$.
        For every $\ell\in\omega^*$, let us now define the sequence of order-compatible nested $\frac{1}{2n}$-chains $\{S^\ell_n\}_{n\in\mathbb{N}}$.
        With respect to what we did in \textbf{STEP 1}, we need to add an intermediate step, defining an opportune sequence of order-compatible nested $\frac{1}{2n}$-chains $\{G^\ell_n\}_{n\in\mathbb{N}}$, applying the inductive hypothesis.
        \\
        Fix $\ell\in\omega^*$.
        We want to define a sequence of order-compatible nested $\frac{1}{2n}$-chains $\{R^\ell_n\}_{n\in\mathbb{N}}$ from $f^{-1}(z_\ell)$ to $z_{\ell+1}$ such that, setting
            $$
            R^\ell = \big(\bigcup_n\widehat{R^\ell_n}\big) \setminus\{f^{-1}(z_\ell),z_{\ell+1}\},
            $$
            we have 
            \begin{equation*}
                \operatorname{otp}(R^\ell,\le_\infty)=
                h_\ell+\xi_\ell+k_\ell.
            \end{equation*}
            Since $(f^{h_\ell-1}(z_\ell),f^{-k_\ell}(z_{\ell+1}))\in\mathcal{M}$ and rank$_{VD}(\xi_\ell)<\alpha$, by applying the inductive hypothesis with $B=B_\ell'$, $p=f^{h_\ell-1}(z_\ell)$ and $q=f^{-k_\ell}(z_{\ell+1})$, there exists a sequence of order-compatible nested $\frac{1}{2n}$-chains $\{G^\ell_n\}_{n\in\mathbb{N}}$ from $f^{h_\ell-1}(z_\ell)$ to $f^{-k_\ell}(z_{\ell+1})$ such that properties A), B), C) and D) hold. 
            In particular, setting 
            $$
            G^\ell=\big(\bigcup_n\widehat{G^\ell_n}\big)\setminus\{f^{h_\ell-1}(z_\ell),f^{-k_\ell}(z_{\ell+1})\},
            $$ 
            condition A) assures that we have $\operatorname{otp}(G^\ell,\le_\infty)=\xi_\ell$, condition D) implies that
            $$
            G^\ell\cap[z_{\ell+1}]=\{f^{-k_\ell-g_\ell}(z_{\ell+1}),\,f^{-k_\ell-g_\ell+1}(z_{\ell+1}),\dots,f^{-k_\ell-1}(z_{\ell+1})\},
            $$
            for some $g_\ell\in\mathbb N_0$ and condition C) implies
            $$
            G^\ell\cap[z_\ell]=\{f^{h_\ell}(z_\ell),\,f^{h_\ell+1}(z_\ell),\dots,\,f^{h_\ell+l_\ell}(z_\ell)\}
            $$
            for some $l_\ell\in\mathbb N_0$.
            Then, for every $n\in\mathbb{N}$ we set
            $$
            R^\ell_n:(f^{-1}(z_\ell),\,z_\ell,\dots,\,f^{h_\ell-2}(z_\ell))\,\dot\sqcup\,G^\ell_n\,\dot\sqcup\,(f^{-k_\ell+1}(z_{\ell+1}),\dots,\,f^{-1}(z_{\ell+1}),\,z_{\ell+1}).
            $$
            Notice that if $h_\ell=0$ (or $k_\ell=0$), we do not need to add points on the left (respectively right) of the chains $G^\ell_n$. 
            Thus, for every $n\in\mathbb{N}$, we define
            $$
            S^\ell_n={R^\ell_n}_{|\text{dom}(R^\ell_n)\setminus\{0,\max(\text{dom}(R^\ell_n)\}},
            $$
            so that
            $$
            R^\ell=\bigcup_n\widehat{S^\ell_n}.
            $$
            Observe that for every $n\in\mathbb{N}$, setting $u^\ell_n=S^\ell_n(0)$ to be the first point of the chain $S^\ell_n$, we have
            \begin{equation}\label{eq_succ_ini}
                d(u^\ell_n,z_\ell)<\frac{1}{2n}.
            \end{equation}
            Indeed, if $h_\ell>0$, by construction, it follows that
            $
            u^\ell_n=R^\ell_n(1)=z_\ell.
            $
            Otherwise, if $h_\ell=0$, we have $u^\ell_n=R^\ell_n(1)=G^\ell_n(1)$ and, since $G^\ell_n$ is a $\frac{1}{2n}$-chain, we find
            $$
            d(u^\ell_n,z_\ell)=d(G^\ell_n(1),f(G^\ell_n(0)))<\frac{1}{2n}.
            $$
            Moreover, for every $n\in\mathbb{N}$, since $R^\ell_n$ is a $\frac{1}{2n}$-chain, setting $v^\ell_n=S^\ell_n(\max(\text{dom}(S^\ell_n)))$ to be the last point of the chain $S^\ell_n$, we have 
            \begin{equation}\label{eq_succ_fin}
                d(f(v^\ell_n),z_{\ell+1})<\frac{1}{2n}.
            \end{equation}

        Let us now define the chain $C_n$ by concatenating the chains $S_n^\ell$ as in \textbf{STEP 2}, that is, for every $n\in\mathbb{N}$ we set
        $$
        C_n:x\,\dot\sqcup\,S^{-n}_n\,\dot\sqcup\,S^{-n+1}_n\,\dot\sqcup\,\dots\,\dot\sqcup\,S^{-1}_n\,\dot\sqcup\,y.
        $$
        By \eqref{eq_zl_ind}, \eqref{eq_succ_ini}, and \eqref{eq_succ_fin}, and employing reasoning analogous to that in \textbf{STEP 2}, it follows that $\{C_n\}_{n\in\mathbb{N}}$ forms a sequence of nested $1/n$-chains from $x$ to $y$.
        Set $C=\big(\bigcup_n\widehat{C_n}\big)\setminus\{x,y\}$. 
        If we show that such a sequence is also order-compatible, then 
        $$
        \operatorname{otp}(C,\le_\infty)=\sum_{\ell\in\omega^*} \xi_\ell'=\sum_{\ell\in\omega^*} h_\ell+\xi_\ell+k_\ell=\xi.
        $$
        We need to check that, for every $n\in\mathbb N$, the chain $C_n$ does not contain any cyclic sub-chain. 
        To do that, it is enough to see that, for every $\ell_1,\ell_2\in\omega^*$ and for every $n\in\mathbb N$,
        \begin{equation}
            \label{condi_acic}
            \left(\bigcup_n\widehat{S^{\ell_1}_n}\right)\cap\left(\bigcup_n\widehat{S^{\ell_2}_n}\right)=\varnothing.
        \end{equation}
        In fact, for every $\ell\in\omega^*$ and for every $n\in\mathbb N$, by the inductive hypothesis, $S_n^\ell$ is acyclic, and so, the only way to have a cyclic sub-chain in $C_m$, for some $m\in\mathbb N$, is that \eqref{condi_acic} is false. 
        Let us now show \eqref{condi_acic}. 
        By construction and by the inductive hypothesis (condition B) namely), for every $\ell\in\omega^*$ and for every $u\in\cup_n\widehat{S^\ell_n}$, we have $[u]\in\operatorname{supp}(B_\ell')\cup[z_{\ell+1}]$, and since the sets supp$(B_\ell')$ are pair-wise disjoint, the only classes in $\mathcal{I}_{/\sim}$ whose points can be used to define more than one sequence of order-compatible nested chains $\{S^\ell_n\}_{n\in\mathbb N}$ are the classes $[z_\ell]$. 
        Now, since for every $\ell,\ell'\in\omega^*$ with $\ell'\ne\ell$ and $\ell'\ne\ell-1$ we have
        $$
        [z_\ell]\cap\left(\bigcup_n\widehat{S^{\ell'}}\right)=\varnothing,
        $$
        it is enough to show that, for every $\ell<-1$,
        \begin{equation}\label{disgiunti}
            \left(\bigcup_n\widehat{S^\ell_n}\right)\cap\left(\bigcup_n\widehat{S^{\ell+1}_n}\right)=\varnothing.
        \end{equation}
        Note that
        $$
        \left(\bigcup_n\widehat{S^\ell_n}\right)\cap\left(\bigcup_n\widehat{S^{\ell+1}_n}\right)\subseteq[z_{\ell+1}]
        $$
        and, since conditions D) and C) respectively assure that
        $$\left(\bigcup_n\widehat{S^\ell_n}\right)\cap[z_{\ell+1}]=\{f^{-k_\ell-g_\ell}(z_{\ell+1}),\,f^{-k_\ell-g_\ell+1}(z_{\ell+1})\dots,\,f^{-1}(z_{\ell+1})\}$$
        and 
        $$\left(\bigcup_n\widehat{S^{\ell+1}_n}\right)\cap[z_{\ell+1}]=\{z_{\ell+1},\,f(z_{\ell+1}),\dots,
        \,
        f^{h_{\ell+1}+l_{\ell+1}}(z_{\ell+1})\}$$
        and, so, equation \eqref{disgiunti} is true.
        Moreover, since supp$(B_\ell')\cap\{[x],\,[y]\}=\varnothing$ for every $\ell\in\omega^*$ by the definition of $B'_\ell$, we have that $\{C_n\}_{n\in\mathbb N}$ is a sequence of order-compatible nested $1/n$-chains from $x$ to $y$ and so condition A) holds.
        \\
        Condition B) is true, too. In fact, since by the inductive hypothesis $\{G_n^\ell\}_{n\in\mathbb{N}}$ consists of points whose classes are in $\operatorname{supp}(B_\ell')\cup\{[z_{\ell+1}]\}$, it follows that the sequence of $\frac{1}{2n}$-chains $\{S_n^\ell\}_{n\in\mathbb{N}}$ consists of points whose classes are in
        $$\operatorname{supp}(B_\ell')\cup\{[z_{\ell+1}]\}.$$
        Therefore, the classes of points used to define the chains $C_n$ are in
        $$
        \operatorname{supp}(A)\cup\{[x],\,[y]\}=\left(\,\bigcup_{\ell<\omega^*}\operatorname{supp}(B_\ell')\,\right)\cup\{[x],\,[y]\}.
        $$
        As seen in \textbf{STEP 2}, where we considered the sum of orders indexed by $\omega^*$, also in this case we have 
        $$
        \widehat{C_n}\cap[x]=\{x\},
        $$
        and so, condition C) is naively true.
        \\
        Finally, we have that, for every $\ell<-1$ and for every $n\in\mathbb N$, by construction,
        $$
        \widehat{S^\ell_n}\cap[y]=\varnothing,
        $$
        because $S^\ell_n$ consists of points with classes in supp$(B_\ell')\cup\{[z_{\ell+1}]\}$ and $$[y]\notin\operatorname{supp}(B_\ell')\cup\{[z_{\ell+1}]\}.$$
        Moreover, by construction, for every $n\in\mathbb N$ and for $\ell=-1$, we have
        $$\widehat{S^{-1}_n}\cap[y]=\{f^{-k_{-1}-g_{-1}}(y),\,f^{-k_{-1}-g_{-1}+1}(y),\dots,\,f^{-1}(y)\}.$$
        Thus, by the definition of $S_n^{-1}$
        condition D) holds. Hence, the sequence of order-compatible nested $1/n$-chains $\{C_n\}_{n\in\mathbb N}$ verifies conditions A), B), C) and D).
        
        \item If $\alpha$ is a limit ordinal, we have 
        $$
        A'=\{B'_i\mid i<\omega\},
        $$
        where $B'_i$ is a $\gamma_i$-structure and $\{\gamma_i\}_{i<\omega}$ is an increasing sequence cofinal in $\alpha$.
        Up to considering a bijection between the set $\{i\mid i<\omega\}$ and the set $\omega^*=\{\dots,\,-n,\dots,\,-2,\,-1\}$, we can enumerate the structures forming $A'$ with $\omega^*$.
        For every $\ell\in \omega^*$, set $\beta_\ell=\text{rank}_{VD}(\xi_\ell')\le\text{rank}_{VD}(\xi_\ell)+1$. Note that, since $\alpha$ is a limit ordinal, we have $\beta_\ell<\alpha$. 
        Thus, there exist $\aleph_0$ ordinal numbers $\gamma_i$ such that $\beta_\ell<\gamma_i<\alpha$.
        Then, it is possible to define an injective function $i(\ell)$ such that, for every  $\ell\in\omega^*$, we have $\gamma_{i(\ell)}>\beta_\ell$.
        Up to considering this injective function $i(\ell)$, we can say that $\text{rank}_{VD}(\xi_\ell')<\gamma_\ell$ for every $\ell\in\omega^*$ and we can index both the family $\{\xi_\ell\}_{\ell\in\omega^*}$ and the family $\{B'_{\ell}\}_{\ell\in\omega^*}$ with $\omega^*$. 
        For every $\ell\in\omega^*$, take a class $[u_\ell]\in B'_\ell$. 
        As before, for every $\ell\in\omega^*$ there exists a point $z_\ell\in[u_\ell]$ such that
        $$
        d(f(x),z_\ell)<-\frac{1}{2\ell}.
        $$
        Set $z_0=y$.
        For every $\ell\in \omega^*$, since $B'_\ell$ is a $\gamma_\ell$-structure with $\gamma_\ell<\alpha$ and rank$_{VD}(\xi'_\ell)=\beta_\ell<\gamma_\ell$, by applying the inductive hypothesis with $p=f^{-1}(z_\ell)$ and $q=z_{\ell+1}$, there is a sequence of order-compatible nested $\frac{1}{2n}$-chains $\{R_n^\ell\}_{n\in\mathbb{N}}$ from $f^{-1}(z_\ell)$ to $z_{\ell+1}$ that satisfies:
        \begin{itemize}
            \item[i)] setting $R^\ell=(\bigcup_n\widehat{R^\ell_n})\setminus\{f^{-1}(z_\ell),z_{\ell+1}\}$, we have that $(R,\le_\infty)$ has order type $\xi_\ell'$;
            \item[ii)] the set $R^\ell$ consists of points whose classes are in $\operatorname{supp}(B_\ell' )\cup\{[z_{\ell+1}]\}$ (recall $[f^{-1}(z_\ell)]\in\operatorname{supp}(B_\ell')$);
            \item[iii)] setting $R^\ell_n: x_0=f^{-1}(z_\ell),\,x_1,\dots,\,x_{m-1},\,x_m=z_{\ell+1}$, if $x_j\in[f^{-1}(z_\ell)]$ for some $j<m$, then $x_{h+1}=f(x_h)\in[f^{-1}(z_\ell)]$ for every $0\leq h<j$;
            \item[iv)]if $x_j\in[z_{\ell+1}]$ for some $j<m$, then $x_h=f(x_{h-1})\in[z_{\ell+1}]$, for every $j<h\leq m$.
        \end{itemize}
        Thus, for every $\ell\in \omega^*$ and $n\in\mathbb{N}$, we define
        $$
        S^\ell_n={R^\ell_n}_{|\operatorname{dom}(R^\ell_n)\setminus\{0,\max(\operatorname{dom}(R^\ell_n))\}}.
        $$
        If we concatenate the chains $S_n^\ell$ as done in \textbf{STEP 2}, we obtain a sequence of order-compatible nested $1/n$-chains $\{C_n\}_{n\in\mathbb{N}}$ from $x$ to $y$.
        By applying reasoning analogous to the case $\alpha=\beta+1$, it follows that the sequence $\{C_n\}_{n\in\mathbb N}$ satisfies conditions A), B), C), and D). 

    \end{itemize}

    Thus, we proved that, for every scattered countably infinite order $\xi$ and for every $(x,y)\in\mathcal{M}$, we can find a sequence of order-compatible nested chains $\{C_n\}_{n\in\mathbb N}$ from $x$ to $y$ such that, setting $C=(\bigcup_n\widehat{C_n})\setminus\{x,y\}$, the order type of $(C,\le_\infty)$ is $\xi$, and so
    $$\xi\in\Omega(x,y).$$
    
\end{proof}

\section*{Acknowledgements}
\textit{The authors thank Alessandro Della Corte for his help.}




\end{document}